\documentclass[11pt, a4paper]{amsart}
\usepackage{amssymb, mathtools, enumerate}
\usepackage[colorlinks=true,allcolors=magenta]{hyperref}
\mathtoolsset{showonlyrefs=true}

\DeclareMathOperator*{\Ker}{Ker}

\DeclareMathOperator*{\R}{Re}

\newcommand{\dd}{\mathrm{d}}
\newcommand{\ep}{\varepsilon}
\newcommand{\B}{\mathcal{B}}

\newcommand{\RR}{\mathbb{R}}
\newcommand{\CC}{\mathbb{C}}
\newcommand{\ZZ}{\mathbb{Z}}
\newcommand{\NN}{\mathbb{N}}

\newcommand{\inv}{^{-1}}
\newcommand{\T}{(T(t))_{t\ge0}}

\newtheorem{thm}{Theorem}[section]
\newtheorem{prp}[thm]{Proposition}

\newtheorem{cor}[thm]{Corollary}
\theoremstyle{definition}

\newtheorem{rems}[thm]{Remarks}

\numberwithin{equation}{section} 

\begin{document}

\title[Polynomial stability of non-linear semigroups]{Polynomial stability of non-linearly damped contraction semigroups}

\author[L.~Paunonen]{Lassi Paunonen}
\address[L.~Paunonen]{Mathematics Research Centre, Tampere University, P.O.~Box 692, 33101 Tampere, Finland}
 \email{lassi.paunonen@tuni.fi}
 
\author[D. Seifert]{David Seifert}
\address[D. Seifert]{School of Mathematics, Statistics and Physics, Newcastle University, Herschel Building, Newcastle upon Tyne, NE1 7RU, United Kingdom}
\email{david.seifert@ncl.ac.uk}

\begin{abstract}
We investigate the stability properties of an abstract class of semi-linear systems. Our main result establishes rational rates of decay for classical solutions assuming a certain non-uniform observability estimate for the linear part and suitable conditions on the non-linearity. We illustrate the strength of our abstract results by applying them to a one-dimensional wave equation with weak non-linear damping and to an Euler--Bernoulli beam with a tip mass subject to non-linear damping.
\end{abstract}



\thanks{This work was supported by the Research Council of Finland Grant number 349002 and by the Heilbronn Institute for Mathematical Research.}

\maketitle

\section{Introduction}\label{sec:int}

Motivated by applications to hyperbolic partial differential equations with non-linear damping, we consider abstractsemi-linear systems of the form
\begin{equation}\label{eq:sys}
\dot x(t)=Ax(t) - B\phi(B^\ast x(t)), \qquad t\ge0,
\end{equation}
to be solved subject to the initial condition $x(0)=x_0$.
Here $A$ is assumed to be the generator of a contraction semigroup $\T$ on a Hilbert space $X$, $x_0\in X$ and $B$ is a bounded linear operator from  another Hilbert space $U$ into $X$. Furthermore, $\phi\colon U\to U$ is a locally Lipschitz continuous map which satisfies $\phi(0)=0$ and is \emph{monotone} in the sense that
$$\R\langle\phi(u_1)-\phi(u_2),u_1-u_2\rangle\ge0,\qquad u_1,u_2\in U.$$  These conditions on $\phi$ allow us to interpret the second summand on the right-hand side of~\eqref{eq:sys} as a non-linear \emph{damping} term.

We investigate the stability properties of the equation~\eqref{eq:sys}. Our main result in Theorem~\ref{thm:main} establishes a rational decay rate for classical solutions of~\eqref{eq:sys} based on a generalised observability-type condition on the operators $A$ and $B$ and the additional assumption that 
\begin{equation}
\label{eq:PhiAssStab}
\R\langle\phi(u),u\rangle \gtrsim
\begin{cases}
 \|u\|^2, &   \|u\|\le\delta,\\
1,  &\|u\|>\delta,
\end{cases}
\end{equation}
for some $\delta>0$. We illustrate our theoretical results in Section~\ref{sec:PDEs} by applying them to a one-dimensional wave equation with weak non-linear damping, and to an Euler--Bernoulli beam  with a tip mass subject to non-linear damping.

The study of polynomial rates of decay for non-linearly damped evolution equations goes back to the work of Haraux and others in the 1970s and 1980s. These early works consider wave equations and abstract second-order systems subject to non-linear viscous damping,  establishing rates of energy decay typically by means of Lyapunov analysis and multiplier methods~\cite{Har87, HarZua88, KakNis79, Nak78, Nak83, Zua88}; we refer the reader to~\cite[Chap.~5 \& 6]{Zua24book} for an overview of several of these developments. More recent works on rational and generalised decay rates for semigroups and (abstract) wave equations with non-linear viscous-type damping include~\cite{Ala10,AlaPri17,ChiKaf25,Kom94a, Nak10}, and the case of non-linear boundary damping has been studied in~\cite{Ala05,LasTat93,VanMar00,VanFer23,XuXu19}. 
Exponential stability in the presence of non-linear damping has been investigated for instance
 in~\cite{ChiMar21,GuiLog19,MarChi20,Teb20}.

One typical feature of these models is that in the special case of linear damping one obtains uniform exponential stability of the underlying semigroup; it is the non-linearity of the damping that gives rise to polynomial (or slower) rates of decay. Our assumptions are fundamentally different. For one thing, in our case the linear system generally fails to be uniformly exponentially stable, and indeed we are interested in cases where the non-linearity \emph{preserves} the polynomial rate of decay one obtains for (classical) solutions of the corresponding linear system. 
 Furthermore, the conditions we impose on the non-linear function $\phi$ are stated in terms of the inner product on the Hilbert space $U$, whereas most previous works in the setting of wave-type equations consider viscous damping acting pointwise in space. As a result, we follow a rather different approach  based on a non-uniform observability estimate.  In doing so we extend to the non-linear setting various results in~\cite{AmmTuc01}; the same results have  previously been extended to abstract wave-type equations with non-linear viscous damping in~\cite{AlaAmm11,AmmBch16,AmmBch17,Phu11}. 
Our results  may be viewed as  providing rates of decay in earlier unquantified studies on asymptotic stability with non-linear damping
such as~\cite{CurOos01, CurZwa16, LasSei03,MarAnd17,Sle89}. 

A function $x(\cdot)\in C^1(\RR_+;X)$ is said to be a  \emph{classical solution} of~\eqref{eq:sys} if  $x(t)\in D(A)$ for all $t\ge 0$ and~\eqref{eq:sys} holds for all $t\geq 0$. 
A function $x(\cdot)\in C(\RR_+;X)$ is said to be a \emph{mild solution} of~\eqref{eq:sys} if 
$$x(t)=T(t)x_0-\int_0^tT(t-s)B\phi(B^*x(s))\,\dd s,\qquad t\ge0.$$
Note that if $\phi$ is a linear map then the notions of classical and mild solutions are consistent with the terminology used for linear abstract Cauchy problems. We recall, moreover, that in the linear case the mild solution is given by the semigroup orbit, and that it is a classical solution precisely when $x_0\in D(A)$. 

If $x$ is a  classical solution of~\eqref{eq:sys}, then dissipativity of $A$ gives
$$
\frac{\dd}{\dd t}\|x(t)\|^2
 = 2\R \langle A x(t) - B\phi(B^\ast x(t)),x(t)\rangle\leq -2\R \langle \phi(B^\ast x(t)),B^\ast x(t)\rangle.
$$
It follows that
\begin{equation}\label{eq:energy_balance}
\|x(t)\|^2 + 2\int_0^t \R \langle \phi(B^\ast x(s)),B^\ast x(s)\rangle\,\dd s 
\le \|x(0)\|^2, \qquad t\ge 0,
\end{equation}
and in particular monotonicity of $\phi$ implies that any classical solution has non-increasing norm. It follows from~\cite[Thm.~11.1.5(b)]{CurZwa20book} that \eqref{eq:sys} has a unique classical solution whenever $x_0\in D(A)$, and by density of $D(A)$ in $X$ there exists a unique mild solution for all $x_0\in X$. In fact, by monotonicity of $\phi$ any two solutions $x_1(\cdot),x_2(\cdot)$ satisfy
\begin{equation}\label{eq:monotone}
\|x_2(t)-x_1(t)\|\le \|x_2(0)-x_1(0)\|, \qquad t\geq 0,
\end{equation}
and for any classical solution $x(\cdot)$ the function $\|\dot x(\cdot)\|$ is non-increasing by~\cite[Cor.~3.7]{Miy92book}.

Given (complex) Hilbert spaces $X$ and $Y$, we write $\B(X,Y)$ for the space of bounded linear operators from $X$ to $Y$, and we write $\B(X)$ for $\B(X,X)$. 
We denote the domain, kernel, spectrum and resolvent set of a linear operator $A$
by  $D(A)$, $\Ker A$, $\sigma(A)$, and $\rho(A)$,  respectively.
If $p$ and $q$ are two real-valued quantities we write $p\lesssim q$ to express that $p\le Cq$ for some constant $C>0$ which is independent of all parameters that are free to vary in a given situation. We shall also make use of standard `big-O' and `little-o'  notation.

\section{Polynomial stability of abstract semi-linear systems}\label{sec:decay}

We now come to our main result. It provides polynomial decay rates for classical solutions of~\eqref{eq:sys} under suitable assumptions on the linear counterpart with $\phi$ equal to the identity function. We write $A_B$ for the corresponding infinitesimal generator $A-BB^*$ with domain $D(A_B)=D(A)$.
The second part of the result shows that if $\phi$ is sufficiently close to linear near the origin, then the decay rate of classical solutions is determined by the rate of resolvent growth in the corresponding linear equation. In this case, even a suboptimal observability inequality~\eqref{eq:ObsEstimate} will, when combined with a sharp resolvent estimate, lead to the optimal decay rate.

\begin{thm}\label{thm:main}
Suppose that $\phi\colon U\to U$ is monotone, locally Lipschitz continuous and satisfies $\phi(0)=0$, and that~\eqref{eq:PhiAssStab}
holds for some $\delta>0$.
Suppose further that there exist $\beta,\tau,c_\tau>0$ such that
\begin{equation}
\label{eq:ObsEstimate}
c_\tau\|(I-A)^{-\beta} x_0\|^2 
\le \int_0^\tau \|B^\ast T(t)x_0\|^2 \,\dd t, \qquad x_0\in X.
\end{equation}
Then
 $i\RR\subseteq\rho(A_B)$. Furthermore,
 all  mild solutions of~\eqref{eq:sys} satisfy $\|x(t)\|\to 0$ as $t\to\infty$, and
 all classical solutions of~\eqref{eq:sys} satisfy $\|x(t)\|=O(t^{-1/(2\beta)})$ as $t\to\infty$.

If, in addition,
$\|(is-A_B)\inv\|\lesssim 1+|s|^\alpha$ for some $\alpha>0$ and all $s\in\RR$, and
 if there exist $\gamma>\alpha/2+1$ and $\kappa,\ep>0$ such that
\begin{equation}
\label{eq:PhiApproximatLinearity}
\|\phi(u)-\kappa u\|\lesssim \|u\|^\gamma, \qquad  \|u\|\le \ep ,
\end{equation}
then $\|x(t)\|= O(t^{-1/\alpha})$ as $t\to\infty$ for any classical solution of~\eqref{eq:sys}.
\end{thm}

\begin{proof}
Let $x(\cdot)$ be a classical solution of~\eqref{eq:sys}, and suppose there exists $t_0\ge0$ such that $\|B^*x(t)\|\le\delta$ for all $t\ge t_0$. Furthermore, let $L_\delta>0$ be such that $\|\phi(u)\|\le L_\delta\|u\|$ for all $u\in U$ with $\|u\|\le\delta$. Let $\tau>0$. For $u\in L^2(0,\tau;U)$ we define 
 $$(\mathbb{F}_\tau u)(t)=\int_0^t B^\ast T(t-s)Bu(s)\,\dd s,\qquad 0\le t\le\tau.$$
Then $\mathbb{F}_\tau\in \B(L^2(0,\tau;U))$ and 
$$B^\ast T(s)x(t) = B^\ast x(t+s)+ (\mathbb{F}_\tau\phi(B^\ast x(t+\cdot)))(s),\qquad 0\le s\le\tau,\ t\ge0,$$
by the variation of parameters formula, and hence
$$
\|B^\ast T(\cdot)x(t)\|_{L^2(0,\tau;U)}\lesssim\|B^\ast x(t+\cdot)\|_{L^2(0,\tau;U)}+\|\phi(B^\ast x(t+\cdot))\|_{L^2(0,\tau;U)}
$$
for all $t\ge0$. For $t\ge t_0$ it follows that 
$$\|B^\ast T(\cdot)x(t)\|_{L^2(0,\tau;U)}\lesssim\|B^\ast x(t+\cdot)\|_{L^2(0,\tau;U)},$$
 and combining this with~\eqref{eq:ObsEstimate} gives
$$\|(I-A)^{-\beta} x(t)\|
\lesssim  \|B^\ast x(t+\cdot)\|_{L^2(0,\tau;U)}, \qquad t\ge t_0.$$
Now let $k\in\NN$ be such that $k\tau\ge t_0$. Using~\eqref{eq:energy_balance} and~\eqref{eq:PhiAssStab} it follows that
$$\begin{aligned}
\|x(k\tau)\|^2&-\|x((k+1)\tau)\|^2= 2\int_0^\tau \R \langle \phi(B^\ast x(k\tau+s)),B^\ast x(k\tau+s)\rangle\, \dd s\\
&\gtrsim \|B^\ast x(k\tau+\cdot))\|_{L^2(0,\tau;U)}^2\gtrsim\|(I-A)^{-\beta}x(k\tau) \|^2,
\end{aligned}
$$
where the implicit constants are independent of $k$. Noting that $x(k\tau)\in D(A)$, we may apply the moment inequality~\cite[Thm.~II.5.34]{EngNag00book} to obtain
$$\|x(k\tau)\|\le \|(I-A)x(k\tau)\|^{\frac{\beta}{1+\beta}}\|(I-A)^{-\beta}x(k\tau)\|^{\frac{1}{1+\beta}}.$$
Thus
$$
\|x(k\tau)\|^2-\|x((k+1)\tau)\|^2\gtrsim\frac{\|x((k+1)\tau)\|^{2(1+\beta)}}{\|(I-A)x(k\tau)\|^{2\beta}},
$$
where we have used the fact that $\|x((k+1)\tau)\|\le\|x(k\tau)\|$. Now
$$\begin{aligned}
\|(I-A)x(k\tau)\|&\le \|x(k\tau)\|+\|Ax(k\tau)-B\phi(B^* x(k\tau))\|+\|B\|\|\phi(B^*x(k\tau))\|\\
&\le (1+L_\delta\|B\|^2)\|x(k\tau)\|+\|\dot{x}(k\tau)\|\\
&\le (1+L_\delta\|B\|^2)\|x(t_0)\|+\|\dot{x}(t_0)\|\\
&=(1+L_\delta\|B\|^2)\|x(t_0)\|+\|Ax(t_0)-B\phi(B^*x(t_0))\|\\
&\le 2(1+L_\delta\|B\|^2)\|x(t_0)\|+\|(I-A)x(t_0)\|\\
&\lesssim\|(I-A)x(t_0)\|,
\end{aligned}
$$
where we have used monotonicity of the functions $\|x(\cdot)\|$ and $\|\dot{x}(\cdot)\|$. It follows that
$$\frac{\|x((k+1)\tau)\|^2}{\|(I-A)x(t_0)\|^2}\le\frac{\|x(k\tau)\|^2}{\|(I-A)x(t_0)\|^2}-c\left(\frac{\|x((k+1)\tau)\|^2}{\|(I-A)x(t_0)\|^2}\right)^{1+\beta}$$
for some constant $c>0$ and all sufficiently large $k\ge1$, and hence
$$\|x(k\tau)\|\lesssim \frac{\|(I-A)x(t_0)\|}{(k\tau+1)^{1/(2\beta)}}$$
for all sufficiently large $k\ge1$ by~\cite[Lem.~1.3.4]{AmmNic15book}. By monotonicity of $\|x(\cdot)\|$ we deduce that 
$$\|x(t)\|\lesssim \frac{\|(I-A)x(t_0)\|}{(t+1)^{1/(2\beta)}}$$
for all sufficiently large $t\ge0$. Next we observe that, for any $\delta\ge\|B^*\|\|x_0\|$, the identity function on $U$ defined by $\phi(u)=u$ for all $u\in U$ satisfies our assumptions  with $t_0=0$ and $L_\delta=1$. We deduce that
$\|T_B(t)(I-A_B)\inv\|=O(t^{-1/(2\beta)})$ as $t\to\infty$, where $(T_B(t))_{t\ge0}$ is the $C_0$-semigroup of contractions generated by $A_B$. It follows from~\cite[Thm.~1.1]{BatDuy08} that $i\RR\subseteq\rho(A_B)$. Moreover, $(T_B(t))_{t\ge0}$ is strongly stable in the sense that $\|T_B(t)x\|\to0$ as $t\to\infty$ for all $x\in X$, by a standard density argument. It then follows from~\cite[Thm.~2.2]{CurZwa16} (noting that the result carries over, with the appropriate modifications, to the setting of complex Hilbert spaces, and that the assumptions of compact resolvent and approximate observability can be replaced by strong stability of $(T_B(t))_{t\ge0}$) that whenever $\phi$ satisfies the conditions of our theorem we have $\|x(t)\|\to0$ as $t\to\infty$ for all  mild solutions of~\eqref{eq:sys}. Hence the first part of the proof shows that $\|x(t)\|=O(t^{-1/(2\beta)})$ as $t\to\infty$ for  all classical solutions of~\eqref{eq:sys}, which completes the proof of the first part of the result.

Now assume, in addition, that
$\|(is-A_B)\inv\|\lesssim 1+|s|^\alpha$ for some $\alpha>0$ and all $s\in\RR$, and that
 there exist $\gamma>\alpha/2+1$ and $\kappa,\ep>0$ such that~\eqref{eq:PhiApproximatLinearity} holds. If $\alpha\ge 2\beta$ then the result already follows without any of the additional assumptions from what has already been proved, so we may assume that $\alpha<2\beta$. Recall from the first part of the proof that $\|x(t)\|\to0$ as $t\to\infty$ for all  mild solutions of~\eqref{eq:sys}, and that $\|x(t)\|=O(t^{-1/(2\beta)})$ as $t\to\infty$ for all classical solutions of~\eqref{eq:sys}. Note also that by~\cite[Lem.~2.11(c)]{ChiPau23} the operator $A_\kappa = A-\kappa BB^\ast$ satisfies $i\RR\subseteq\rho(A_\kappa)$ and $\|T_\kappa(t)A_\kappa\inv\|=O(t^{-1/\alpha})$ as $t\to\infty$, where $(T_\kappa(t))_{t\ge0}$ is the contraction semigroup  generated by $A_\kappa$. Let $x(\cdot)$ be a classical solution of~\eqref{eq:sys}, and let $t_0\ge0$ be such that $\|B^*x(t)\|\le\ep$ for all $t\ge t_0$.  We have
 $$\dot{x}(t)=(A-\kappa BB^\ast) x(t) + By(t),\qquad t\ge 0,$$
 where $y(t)=\kappa B^\ast x(t)-\phi(B^\ast x(t))$ for all $t\ge 0$. Note that, by~\eqref{eq:PhiApproximatLinearity} and boundedness of $B$, we have $\|y(t)\|\lesssim \| x(t)\|^{\gamma}$ for $t\ge t_0$. Given $t\ge0$ let $t_\theta=\max\{t_0,t-t^\theta/2\}$, where $\theta\in(0,1]$ is to be chosen later, noting that $t_\theta\to\infty$ and $t-t_\theta\to\infty$ as $t\to\infty$. The variation of parameters formula gives
 \begin{equation}\label{eq:three}
 \begin{aligned}
 x(t)
 = T_\kappa (t-t_0)x(t_0) 
&+ T_\kappa (t-t_\theta)\int_{t_0}^{t_\theta} T_\kappa (t_\theta-s)By(s)\,\dd s\\
\qquad &+ \int_{t_\theta}^t T_\kappa (t-s)By(s)\,\dd s
 \end{aligned}
 \end{equation}
for all $t\ge t_0$.  Since $x(t)\in D(A)$ for all $t\ge 0$, we have  $\|T_\kappa (t-t_0)x(t_0)\| =O(  t^{-1/\alpha})$ as $t\to\infty$. 
 Moreover, since $\phi$ is locally Lipschitz continuous and $Bx(\cdot)$ is continuously differentiable, we have $\int_{t_0}^{t_\theta} T_\kappa (t_\theta-s)By(s)\,\dd s\in D(A)$ for all $t\ge t_0$ by~\cite[Cor.~3.1.17]{AreBat11book}.
The second term in~\eqref{eq:three} therefore satisfies
$$\begin{aligned}
\bigg\|T_\kappa (t-t_\theta)\int_{t_0}^{t_\theta} T_\kappa (t_\theta&-s)By(s)\,\dd s\bigg\|\\ &\le\|T_\kappa (t-t_\theta)A_\kappa\inv\|\bigg\|A_\kappa\int_{t_0}^{t_\theta} T_\kappa (t_\theta-s)By(s)\,\dd s\bigg\|
\end{aligned}$$
 for all $t\ge t_\theta$. Now $\|T_\kappa (t-t_\theta)\smash{A_\kappa\inv}\|=O(t^{-\theta/\alpha})$ as $t\to\infty$. On the other hand,  we have 
$$ \int_{t_0}^{t_\theta} T_\kappa (t-s)By(s)\,\dd s = x(t_\theta)-T_\kappa (t_\theta-t_0)x(t_0),\qquad t\ge t_0,$$ 
by the variation of parameters formula. Thus 
$$\begin{aligned}
\bigg\|A_\kappa&\int_{t_0}^{t_\theta} T_\kappa (t_\theta-s)By(s)\,\dd s\bigg\|\le \|A_\kappa x(t_\theta)\|+\|A_\kappa T_\kappa (t_\theta-t_0)x(t_0)\|\\
&\le  \|Ax(t_\theta) - B\phi(B^\ast x(t_\theta))\|+\|By(t_\theta)\|+\|T_\kappa (t_\theta-t_0)A_\kappa x(t_0)\|\\
&\lesssim \|\dot{x}(t_0)\|+\|x(t_0)\|^\gamma+\|A_\kappa x(t_0)\|
\end{aligned}$$
for all $t\ge t_0$, where we have used contractivity of $(T_\kappa(t))_{t\ge0}$ and the fact that both $\|x(\cdot)\|$ and $\|\dot x(\cdot)\|$ are non-increasing functions. Thus the second term in~\eqref{eq:three} satisfies
$$\bigg\|T_\kappa (t-t_\theta)\int_{t_0}^{t_\theta} T_\kappa (t_\theta-s)By(s)\,\dd s\bigg\|=O({t^{-\theta/\alpha}}),\qquad t\to\infty.$$
The third and final term in~\eqref{eq:three} satisfies
$$ \int_{t_\theta}^t T_\kappa (t-s)By(s)\,\dd s
= \int_{0}^t T_\kappa (t-s)By_\theta(s)\,\dd s,$$
where $y_\theta(t)=0$ for $t\in[0,t_\theta)$ and $y_\theta(t)=y(t)$ for $t\ge t_\theta$. By~\cite[Lem.~2.2.6]{Oos00book} the maps $\Phi_t\in\B(L^2(0,t;U),X)$, defined by $\Phi_t u=\int_0^t T_\kappa (t-s)Bu(s)\,\dd s$ for $t\ge 0$ and $u\in L^2(0,t;U)$, have uniformly bounded operator norms; see also~\cite[Thm.~6.5.6]{CurZwa20book}, \cite[Cor.~6.1]{Sta02}. Since $\|y(t)\|\lesssim \|x(t)\|^\gamma$ for all $t\ge t_0$, and since $\|x(t)\|=O(t^{-1/(2\beta)})$ as $t\to\infty$ by the first part of the result, we deduce that
$$\bigg\| \int_{t_\theta}^t T_\kappa (t-s)By(s)\,\dd s\bigg\|^2\lesssim  \int_{t_\theta}^t \|y(s)\|^2\,\dd s\lesssim\int_{t_\theta}^t\frac{\dd s}{(1+s)^{\gamma/\beta}}\le\frac{t-t_\theta}{(1+t_\theta)^{\gamma/\beta}}$$
for all $t\ge t_0$.
It follows that the three terms in~\eqref{eq:three} are of order
$O(t^{-1/\alpha})$, $O(t^{-\theta/\alpha})$ and  $O(t^{-\mu})$, respectively, as $t\to\infty$, where  $\mu=\gamma/(2\beta)-\theta/2$. If  $\beta\le \alpha\gamma/(\alpha+2)$
we set $\theta=1$. Then $\mu\ge1/\alpha$, so the third term decays at least as fast as the first two, and we obtain $\|x(t)\|=O(t^{-1/\alpha})$ as $t\to\infty$, which is what we wanted to prove. On the other hand, if $\beta> \alpha\gamma/(\alpha+2)$ we set $\theta=\alpha\gamma/(\beta(\alpha+2))$. Then $\theta\in(0,1)$ and the rate is determined by the second and third terms, which decay at the same speed, giving $\|x(t)\|=O(t^{-1/(2\sigma\beta)})$ as $t\to\infty$, where $\sigma=(\alpha+2)/(2\gamma)$. The latter estimate is strictly worse than the one we wish to prove. On the other hand, since $\gamma>\alpha/2+1$ by assumption, we have $\sigma\in(0,1)$ and hence this decay rate is strictly \emph{better} than the estimate $\|x(t)\|=O(t^{-1/(2\beta)})$ as $t\to\infty$ coming from the first part of the result. In this second case, we may therefore repeat the argument with $\beta$ replaced by $\sigma\beta$. As above, we find that if $\sigma\beta\le \alpha\gamma/(\alpha+2)$ then $\|x(t)\|=O(t^{-1/\alpha})$ as $t\to\infty$ and the proof is complete, while if $\sigma\beta> \alpha\gamma/(\alpha+2)$ then  $\|x(t)\|=O\smash{(t^{-1/(2\sigma^2\beta)})}$ as $t\to\infty$. If necessary we may now iterate this process, terminating after $k\ge0$ repetitions if $\sigma^k\beta\le \alpha\gamma/(\alpha+2)$, giving the desired the rate $\|x(t)\|=O(t^{-1/\alpha})$ as $t\to\infty$, and otherwise replacing $\sigma^k\beta$ by $\sigma^{k+1}\beta$. Since $\sigma^k\beta\le \alpha\gamma/(\alpha+2)$ for all sufficiently large $k\ge0$ by virtue of the fact that $\sigma\in(0,1)$, the process must eventually terminate yielding $\|x(t)\|=O(t^{-1/\alpha})$ as $t\to\infty$, as required.
\end{proof}

\begin{rems}
\begin{enumerate}[(a)]
\item If~\eqref{eq:PhiApproximatLinearity} holds then necessarily $\phi(0)=0$, so this condition can be omitted whenever  we assume~\eqref{eq:PhiApproximatLinearity}. It moreover follows from~\eqref{eq:PhiApproximatLinearity} that $\phi$ is differentiable at zero, with derivative $D\phi(0)=\kappa I$, and that $\R\langle\phi(u),u\rangle\gtrsim\|u\|^2$ for all $u\in U$ of sufficiently small norm. The latter implies that whenever we are assuming~\eqref{eq:PhiApproximatLinearity} we need only verify the second part of~\eqref{eq:PhiAssStab}.
\item The class of functions $\phi\colon U\to U$ to which the Theorem~\ref{thm:main} can be applied includes \emph{radial} functions defined by  $\phi(0)=0$ and 
$$\phi(u)=\psi(\|u\|)\frac{u}{\|u\|},\qquad u\in U\setminus\{0\},$$
 where $\psi\colon\RR_+\to\RR_+$ is a non-decreasing and locally Lipschitz continuous function satisfying $\psi(0)=0$. Then $\phi$ is monotone and locally Lipschitz continuous, and~\eqref{eq:PhiAssStab} holds provided there exists $\delta>0$ such that $\psi(r)\gtrsim r$ for  $0\le r\le\delta$. Furthermore, if $\psi$ is twice continuously differentiable with $\psi'(0)>0$ a simple estimate using Taylor's theorem shows that~\eqref{eq:PhiApproximatLinearity} holds with $\kappa=\psi'(0)$ for some $\ep>0$ and some $\gamma\ge2$. A concrete example is the function $\psi$ defined by $\psi(r)=\tanh(r)$ for $r\ge0$. In this case, there exists $\ep>0$ such that~\eqref{eq:PhiApproximatLinearity} holds for $\kappa=1$ and $\gamma=3$.
\item If in addition to the assumptions made in the first part of Theorem~\ref{thm:main} we assume that $\phi$ is \emph{globally} Lipschitz continuous and satisfies $\R\langle\phi(u),u\rangle\gtrsim  \|u\|^2$ for \emph{all} $u\in U$, then a slight adaptation of the proof yields \emph{uniform} polynomial decay of the form 
$$\|x(t)\| \lesssim \frac{\|x(0)\|_{D(A)}}{(1+t)^{1/(2\beta)}} , \qquad t\ge 0,$$
for all classical solutions of~\eqref{eq:sys}, where $\|\cdot\|_{D(A)}$ denotes the graph norm of $A$.
This conclusion is analogous to earlier results showing that exponential stability may be preserved  under global
growth conditions on the non-linear damping; see for instance~\cite{LasTat93}.
\item It is clear that the above proof relies crucially on boundedness of $B$. On the other hand, the results in~\cite{Ala05, AmmTuc01, ChiPau23} allow for certain `unbounded' operators $B$ mapping into an extrapolation space of $X$. We leave open to what extent our results can be generalised to this considerably more challenging setting.
\end{enumerate}
\end{rems}

As a simple consequence of Theorem~\ref{thm:main} we obtain the following improvement of~\cite[Thm.~4.4]{ChiPau23} and~\cite[Thm.~4.2]{PauSei25a} in the linear case. Both of these earlier results include the assumption that $\beta\in(0,1]$, and in~\cite[Thm.~4.4]{ChiPau23} the further assumption that $D(A)=D(A^*)$ is added. Our present approach requires neither. As before, we assume that $A$ is the generator of a contraction semigroup $\T$ on a Hilbert space $X$, and that $B\in \B(U,X)$ for some other Hilbert space $U$. Moreover, we let $A_B=A-BB^*$ with domain $D(A_B)=D(A)$, and we let $(T_B(t))_{t\ge0}$ denote the contraction semigroup on $X$ generated by $A_B$.

\begin{cor}
\label{cor:LinearABBimproved}
 Suppose that there exist $\beta,\tau,c_\tau>0$ such that
$$c_\tau\|(I-A)^{-\beta} x_0\|^2 
\le \int_0^\tau \|B^\ast T(t)x_0\|^2 \,\dd t, \qquad x_0\in X.$$
Then
 $i\RR\subseteq\rho(A_B)$, $\|T_B(t)x\|\to 0$ as $t\to\infty$ for all $x\in X$, and $\|T_B(t)x\|=o(t^{-1/(2\beta)})$ as $t\to\infty$ for all $x\in D(A)$.
\end{cor}

\begin{proof}
Applying Theorem~\ref{thm:main} with $\phi\colon U\to U$ taken to be the identity map we obtain $i\RR\subseteq \rho(A_B)$ and  $\|T_B(t)x_0\|=O(t^{-1/(2\beta)})$ as $t\to\infty$ for all $x_0\in D(A)$. Since the semigroup $(T_B(t))_{t\ge0}$ is contractive and $D(A)$ is dense in $X$ a simple approximation argument shows that $\|T_B(t)x_0\|\to0$ as $t\to\infty$ for all $x_0\in X$. Moreover, an application of the uniform boundedness theorem yields $\|T_B(t)\smash{A_B\inv}\|=O(t^{-1/(2\beta)})$ as $t\to\infty$, and hence $\|T_B(t)x_0\|=o(t^{-1/(2\beta)})$ as $t\to\infty$ for all $x_0\in D(A)$ by~\cite[Thm.~2.4]{BorTom10}.
\end{proof}

We conclude this section by showing that resolvent estimates for $A$ along the imaginary axis provide a sufficient condition for the observability estimate~\eqref{eq:ObsEstimate} to hold, at least in the special case when $A$ is skew-adjoint and has uniform spectral gap.

\begin{prp}\label{prp:obs}
Suppose that $A$ is skew-adjoint and that $\sigma(A)$ consists of simple and uniformly separated eigenvalues. Suppose further that $i\RR\subseteq\rho(A_B)$ and that $\|(is-A_B)\inv\|\lesssim 1+|s|^\alpha$ for some $\alpha>0$ and all $s\in\RR$. Then there exist $\tau,c_\tau>0$ such that~\eqref{eq:ObsEstimate} holds with $\beta=\alpha/2$.
\end{prp}

\begin{proof}
Let us denote the eigenvalues of $A$ by $is_k$, $k\ge1$, where $s_k\in\RR$ for all $k$. By assumption there exists $\delta>0$ such that $|s_k-s_j|\ge\delta$ for all $j,k\ge1$ with $j\ne k$. Let $\beta=\alpha/2$. By~\cite[Prop.~5.1]{ChiPau23}  we have
\begin{equation}
\label{eq:WavepacketSimple}
\|B^\ast x\|\gtrsim \frac{\|x\|}{1+|s_k|^\beta} 
\end{equation}
for all $k\ge1$ and $x\in \Ker(is_k-A)$. For $k\ge1$, let $e_k$ be a normalised eigenvector corresponding to the eigenvalue $is_k$. Then $\{e_k:k\ge1\}$ is an orthonormal basis for $X$, and it follows from~\eqref{eq:WavepacketSimple} that  $(1+|s_k|^2)^\beta\|B^\ast e_k\|^2\gtrsim1$ for all $k\ge1$. Let $\tau>2\pi/\delta$. It follows from Ingham's inequality~\cite[Prop.~1.5.2]{AmmNic15book} that
$$\begin{aligned}
\int_0^\tau \|B^\ast T(t)x_0\|^2 \,\dd t&= \int_0^\tau \bigg\|\sum_{k=1}^\infty e^{is_kt}\langle x_0,e_k\rangle B^\ast e_k\bigg\|^2\, \dd t\\
&\gtrsim \sum_{k=1}^\infty |\langle x_0,e_k\rangle|^2 \|B^\ast e_k\|^2\\
&=\sum_{k=1}^\infty |\langle (I-A)^{-\beta}x_0,e_k\rangle|^2 (1+|s_k|^2)^\beta\|B^\ast e_k\|^2\\
&\gtrsim \sum_{k=1}^\infty |\langle (I-A)^{-\beta}x_0,e_k\rangle|^2=\|(I-A)^{-\beta}x_0\|^2
\end{aligned}
$$
for all $x_0\in X$, and the proof is complete.
\end{proof}

\begin{rems}
\begin{enumerate}[(a)]
\item The estimate in~\eqref{eq:WavepacketSimple} is sometimes referred to as a \emph{wave\-packet condition}. The above proof shows that, under a uniform spectral gap condition, we may pass from a resolvent estimate for $A_B$ to a wave\-packet condition and thence to an observability estimate. In~\cite[Sect.~3]{ChiPau23} it is shown, conversely, that even without the spectral gap condition general wavepacket conditions of this kind can be used to obtain resolvent estimates for $A_B$, which in turn imply decay rates for classical orbits of the damped semigroup  $(T_B(t))_{t\ge0}$.
\item By applying more sophisticated versions of Ingham's inequality such as~\cite[Cor.~1.5.4]{AmmNic15book} we may weaken the uniform gap condition to allow  for a degree of repetition and clustering in the  eigenvalues of $A$.
\end{enumerate}
\end{rems}

Combining Proposition~\ref{prp:obs} with Theorem~\ref{thm:main} gives the following result.

\begin{cor}\label{cor:gap}
Suppose that $A$ is skew-adjoint,  $\sigma(A)$ consists of simple and uniformly separated eigenvalues,  and that $i\RR\subseteq\rho(A_B)$ and $\|(is-A_B)\inv\|\lesssim 1+|s|^\alpha$ for some $\alpha>0$ and all $s\in\RR$. Suppose further  that $\phi\colon U\to U$ is monotone and locally Lipschitz continuous and that~\eqref{eq:PhiAssStab} holds for some $\delta>0$.
Then all  mild solutions of~\eqref{eq:sys} satisfy $\|x(t)\|\to 0$ as $t\to\infty$, and all classical solutions of~\eqref{eq:sys} satisfy $\|x(t)\|= O(t^{-1/\alpha})$ as $t\to\infty$.
\end{cor}

\section{Applications to non-linearly damped evolution equations}\label{sec:app}
\label{sec:PDEs}

In this section we illustrate the strength of our main result in two concrete applications, namely a one-dimensional wave equation with weak non-linear damping and an Euler--Bernoulli beam  with a tip mass subject to non-linear damping. In both cases the eigenvalues are simple and uniformly separated, so we may apply Corollary~\ref{cor:gap}. We emphasise, however, that our main result, Theorem~\ref{thm:main}, is applicable much more generally provided one is able to obtain a non-uniform observability estimate as in~\eqref{eq:ObsEstimate}. For an interesting potential application in a setting where there is no uniform spectral gap we refer the reader to the system studied in~\cite{SuTuc20}, which models the dynamics of small-amplitude water waves.

\subsection{The wave equation with weak non-linear damping} Consider the wave equation on the unit interval subject to weak non-linear damping, namely
$$u_{tt}(x,t)=u_{xx}(x,t)+b(x)\phi\left(\int_0^1b(s)u_t(s,t)\,\dd s\right),\qquad x\in(0,1),\ t>0,$$
to be solved subject to the boundary conditions $u(0,t)=u(1,t)=0$ for all $t>0$ and the initial conditions  $u(\cdot,0)\in H_0^1(0,1)$, $u_t(\cdot,0)\in L^2(0,1)$. Here the function $b\in L^2(0,1;\RR)$ models the presence of weak (distributed) damping and the function $\phi\colon\CC\to\CC$ is potentially non-linear. We may formulate the problem in the form of~\eqref{eq:sys} for the state variable $x(t)=(u(\cdot,t),u_t(\cdot,t))$, $t\ge0$,  by setting $X=H_0^1(0,1)\times L^2(0,1)$, $U=\CC$, choosing $A$ to be the operator defined by $A(u,v)=(v,u'')$ for all $(u,v)$ in the domain $D(A)=(H^2(0,1)\cap H_0^1(0,1))\times H_0^1(0,1)$ and defining $B\in\B(U,X)$ by $Bz=(0,b(\cdot)z)$ for all $z\in\CC$. We denote the (rescaled) Fourier sine series coefficients of $b$ by
$$b_n=\int_0^1 b(x)\sin(n \pi x)\,\dd x,\qquad n\ge1.$$

\begin{prp}\label{prp:weak}
Consider the system~\eqref{eq:sys} for the weakly damped wave equation as above, and suppose there exists $\beta>0$ such that $|b_n|\gtrsim n^{-\beta}$ for all $n\ge1$. Suppose further that $\phi\colon U\to U$ is monotone and locally Lipschitz continuous, and that~\eqref{eq:PhiAssStab} holds for some $\delta>0$. Then all  mild solutions of~\eqref{eq:sys} satisfy $\|x(t)\|\to 0$ as $t\to\infty$, and all classical solutions satisfy $\|x(t)\|= O(t^{-1/(2\beta)})$ as $t\to\infty$.
\end{prp}

\begin{proof}
The operator $A$ is skew-adjoint, and its spectrum consists of the simple eigenvalues $i\pi n$ for $n\in\ZZ\setminus\{0\}$, which are uniformly separated. As in the proof of~\cite[Cor.~6.3]{ChiPau23}, our assumption on the decay of the Fourier sine series coefficients of $b$ implies that $i\RR\subseteq\rho(A_B)$ and $\|(is-A_B)\inv\|\lesssim 1+|s|^{2\beta}$ for all $s\in\RR$. The result now follows from Corollary~\ref{cor:gap}.
\end{proof}

\begin{rems}
\begin{enumerate}[(a)]
\item Note that in the particular case where $\phi$ is the identity map on $\CC$, so that~\eqref{eq:PhiAssStab} holds for $\delta=1$,  Proposition~\ref{prp:weak} is sharp in the sense that if $\limsup_{n\to\infty}n^\beta|b_n|>0$ then for any function $r\colon\RR_+\to(0,\infty)$ such that $r(t)=o(t^{-1/(2\beta)})$ as $t\to\infty$ there exists a classical solution such that $\|x(t)\|\ne O(r(t))$ as $t\to\infty$. Indeed, if this is not the case then an application of the uniform boundedness principle gives $\|T_B(t)\smash{A_B^{-1}}\|=O(r(t))$ as $t\to\infty$. On the other hand, since $\limsup_{n\to\infty}n^\beta|b_n|>0$ it follows from~\cite[Prop.~5.1]{ChiPau23} that 
$$\limsup_{|s|\to\infty}\frac{\|(is-A_B)^{-1}\|}{|s|^{2\beta}}>0.$$
 Now~\cite[Prop.~5.3]{ChiPau23} implies that $\limsup_{t\to\infty}t^{1/(2\beta)}\|T_B(t)\smash{A_B\inv}\|>0$, yielding the required contradiction.
\item Note that since $b\in L^2(0,1;\RR)$ we must have $(b_n)_{n\ge1}\in\ell^2$, so necessarily $\beta>1/2$ in Proposition~\ref{prp:weak}. On the other hand, every $\beta>1/2$ can be achieved, for instance by considering the function $b$ whose Fourier sine series coefficients are $b_n=n^{-\beta}$ for $n\ge1$. We refer the interested reader to~\cite[Rem.~6.4]{ChiPau23} for a discussion on the possibility of achieving these decay rates by means of functions $b$ that possess additional regularity.
\end{enumerate}
\end{rems}

\subsection{The SCOLE model with non-linear damping}

In this section we analyse the stability of the SCOLE model~\cite{GuoIva05, TayBal88},
 which consists of an Euler--Bernoulli beam equation coupled with an ODE modelling the dynamics of a \emph{tip mass}.
The system has the form
$$
\left\{
\begin{aligned}
\rho(x)u_{tt}(x,t)&=-(E\!I(x)u_{xx}(x,t))_{xx}, &x\in(0,1),\ t>0,\\
mu_{tt}(1,t)-(E\!Iu_{xx})_x(1,t)&=-\phi_1(u_t(1,t),u_{xt}(1,t)), &t>0,\\
Ju_{xtt}(1,t)+E\!I(1)u_{xx}(1,t)&=-\phi_2(u_t(1,t),u_{xt}(1,t)), &t>0,
\end{aligned}
\right.$$
to be solved subject to the boundary conditions $u(0,t)=u_x(1,t)=0$ for all $t>0$ and the initial conditions  $u(\cdot,0)\in H_L^2(0,1)$, $u_t(\cdot,0)\in L^2(0,1)$, where $H_L^2(0,1)=\{u\in H^2(0,1) : u(0)=u'(0)=0\}$. In this model, $w(x,t)$ is the deflection of the beam at $x\in [0,1]$ and  time $t\ge 0$,  $E\!I\in C^4([0,1])$ and $\rho\in C^4([0,1])$ are the (uniformly positive) flexural rigidity and mass density of the beam, respectively, and $m,J>0$ are, respectively, the mass and moment of inertia of the tip mass. Finally, the two functions $\phi_1,\phi_2\colon \CC^2\to \CC$ describe the non-linear effects of the boundary condition at $x=1$. We assume that there exist $a,b>0$ and
$\zeta\in C^2([0,1])$ such that $\zeta(0)=0$ and
$$
\begin{aligned}
2(1-a)\rho(x)-(\rho\zeta)'(x)<-b,\\
E\!I(x)(1-a-2\zeta'(x))+\frac{1}{2}(E\!I\,\zeta)'(x)<-b
\end{aligned}$$
for all $x\in [0,1]$.
Note that these conditions are in particular satisfied if $E\!I$ and $\rho$ are constant functions, in which case we may take $\zeta(x) = 2x$ for $0\le x\le1$.

We  may formulate the problem in the form of~\eqref{eq:sys} for the state variable $x(t)=(u(\cdot,t),u_t(\cdot,t),u_t(1,t),u_{xt}(1,t))$, $t\ge0$,  by setting $X=H_L^2(0,1)\times L^2(0,1)\times\CC^2$,
 $U=\CC^2$, defining the operator $A$ by $$A(u,v,\lambda,\mu)=\big(v,-\rho\inv (E\!Iu'')'',m\inv (E\!Iu'')'(1), -J\inv E\!I(1)u''(1)\big)$$ for all $(u,v,\lambda,\mu)$ in the domain 
$$D(A) = \big\{(u,v,\lambda,\mu)\in H^4(0,1)\times H_L^2(0,1)\times \CC^2 : v(1)=\lambda, \ v'(1)= \mu \big\},$$ the operator $B\in \B(U,X)$ by $B(\lambda,\mu)=(0,0,\lambda,\mu)$ for all $(\lambda,\mu)\in U$, and the map $\phi\colon U\to U$ by $\phi(\lambda,\mu)=(m\inv\phi_1(\lambda,\mu),J\inv\phi_2(\lambda,\mu))$ for all $(\lambda,\mu)\in U$. Note that $\phi$ is locally Lipschitz continuous if and only if both $\phi_1$ and $\phi_2$ are locally Lipschitz continuous. The function $\phi$ is monotone for instance if $\phi_1$ is independent of the second variable, $\phi_2$ is independent of the first variable and both of the maps $\lambda\mapsto\phi_1(\lambda,0)$ and $ \mu\mapsto\phi_2(0,\mu)$ are monotone. We obtain the following polynomial stability result.

\begin{prp}\label{prp:SCOLE}
Consider the system~\eqref{eq:sys} for the SCOLE model with non-linear damping as above. 
Suppose  that $\phi \colon U\to U$ is monotone and locally Lipschitz continuous, and that~\eqref{eq:PhiAssStab} holds for some $\delta>0$.
 Then all  mild solutions of~\eqref{eq:sys} satisfy $\|x(t)\|\to 0$ as $t\to\infty$, and all classical solutions satisfy $\|x(t)\|= O(t^{-1/2})$ as $t\to\infty$.
\end{prp}

\begin{proof}
By~\cite[Prop.~1.1]{GuoIva05} the operator $A$ is skew-adjoint, has compact resolvent, and its eigenvalues are simple.
In addition, it follows from~\cite[Prop.~1.2]{GuoIva05} that the eigenvalues of $A$ are uniformly separated.
Finally, $i\RR\subseteq \rho(A_B)$ and $\|(is-A_B)^{-1}\|\lesssim 1+s^2$ for $s\in\RR$  by~\cite[Thm.~3.1]{FkiPau25arxiv} and~\cite[Prop.~1.3]{BatDuy08}, so result follows from Corollary~\ref{cor:gap}.
\end{proof}

\bibliographystyle{plain}

\end{document}